\documentclass[11pt, a4paper]{amsart}

\usepackage[ansinew]{inputenc}
\usepackage[all]{xy}
\SelectTips{cm}{}
\usepackage[%pagebackref,%colorlinks,linkcolor=black,citecolor=black%
  ,pdfborder={0 0 0}
  ,urlcolor=black,a4paper,hypertexnames=false]{hyperref}
\hypersetup{pdfauthor={pietro_capovilla}
  ,pdftitle=rel_vanishing_theorem}

\usepackage{amsfonts}
\usepackage{amssymb, amsmath,amsthm, mathtools}
\usepackage{tabularx, hyperref}
\usepackage{enumerate, comment}

\usepackage{tikz-cd}

\usepackage[textsize=tiny]{todonotes}
\setuptodonotes{fancyline, color=blue!30}

\makeatletter
\newtheorem*{rep@theorem}{\rep@title}
\newcommand{\newreptheorem}[2]{%
\newenvironment{rep#1}[1]{%
 \def\rep@title{#2 \ref{##1}}%
 \begin{rep@theorem}}%
 {\end{rep@theorem}}}
\makeatother

\newtheorem{intro_thm}{Theorem}

\newtheorem{lemma}{Lemma}[section]

\newreptheorem{thm}{Theorem}  %for repeating the theorem
\newtheorem{prop}[lemma]{Proposition}
\newreptheorem{prop}{Proposition}  %for repeating the proposition

\newreptheorem{cor}{Corollary}  %for repeating the corollary

\theoremstyle{definition}
\newtheorem{defn}[lemma]{Definition}

\newreptheorem{quest}{Question}  %for repeating the question

\newtheorem{example}[lemma]{Example}
\newtheorem{rem}[lemma]{Remark}

\theoremstyle{definition}

\numberwithin{equation}{section}

\renewcommand{\epsilon}{\varepsilon}

\renewcommand{\hat}{\widehat}

\renewcommand{\phi}{\varphi}

\newcommand{\N}{\mathbb{N}}

\newcommand{\R}{\mathbb{R}}

\newcommand{\comp}{\mathrm{comp}}

\newcommand{\alt}{\textup{alt}}

\newcommand{\acts}{\curvearrowright}

\newcommand{\Aut}{\textup{Aut}}
\newcommand{\K}{\mathcal{K}}

\newcommand{\Elle}{\mathcal{L}}

\newcommand{\A}{\mathcal{A}}
\newcommand{\U}{\mathcal{U}}
\newcommand{\V}{\mathcal{V}}

\newcommand{\mult}{\mathrm{mult}}

\newcommand{\im}{\mathrm{im}}

\newcommand{\twprod}{\mathbin{%
		\ooalign{\raise1.15ex\hbox{$\scriptstyle\sim$}\cr\hidewidth$\times$\hidewidth\cr}%
}}
\newcommand{\twprodboth}{\mathbin{%
		\ooalign{\raise1.15ex\hbox{$\scriptstyle(\sim)$}\cr\hidewidth$\times$\hidewidth\cr}%
}}

\DeclareMathAlphabet{\mathcal}{OMS}{cmsy}{m}{n}

\long\def\forget#1{}

\makeatletter
\@namedef{subjclassname@2020}{%
  \textup{2020} Mathematics Subject Classification}
\makeatother

\begin{document}

\title{Amenable covers and relative bounded cohomology}

\author[]{Pietro Capovilla}
\address{Scuola Normale Superiore, Pisa Italy}
\email{pietro.capovilla@sns.it}

\thanks{}

%55N10          Singular theory
%55U10			Simplicial sets
%57N65			Algebraic topology of manifolds
\begin{abstract}
	We establish a relative version of Gromov's Vanishing Theorem in the presence of amenable open covers with small multiplicity, extending a result of Li, L{\"o}h and Moraschini. Our approach relies on Gromov's theory of multicomplexes.
\end{abstract}
\maketitle
\section{Introduction}
Let $X$ be a topological space.
Gromov's Vanishing Theorem \cite{gromov1982volume} affirms that, if $X$ admits an amenable open cover of multiplicity $k$, then the comparison map $H^n_b(X;\R)\rightarrow H^n(X;\R)$ from bounded cohomology to singular cohomology of $X$ vanishes in every degree $n\geq k$.
Several proofs of this result are available in the literature, whose techniques range between Gromov's theory of multicomplexes \cite{gromov1982volume, frigerio2022amenable, FM23}, sheaf theory \cite{ivanov1985foundations, Iva17}, equivariant nerves and classifying spaces for families \cite{LS20} and homotopy theory \cite{Rap24}.
A generalization of Gromov's result to the relative setting has been provided by Li, L{\"o}h and Moraschini in \cite{li2022bounded} and by Raptis in \cite{Rap24}.
In this paper we extend the result of Li, L{\"o}h and Moraschini in two directions, following Gromov's approach via multicomplexes.
We refer the reader to Remark \ref{rem: connection with li loeh and moraschini} for a detailed discussion on the relationship between our results and the ones in \cite{li2022bounded}.
A pair of topological spaces $(X,A)$ is called \emph{triangulable} if there exists a pair of simplicial complexes $(T,S)$ such that $(X,A)=(|T|,|S|)$.
\begin{intro_thm}
	\label{thm: main theorem}
	Let $(X,A)$ be a triangulable pair and assume that the kernel of the morphism $\pi_1(A\hookrightarrow X, x)$ is amenable for every $x \in A$.
	Let $\U$ be an amenable cover of $X$ by path-connected open subsets such that:
	\begin{enumerate}
		\item[$(\mathrm{RC1})$] For every $U\in \U$, $U\cap A$ is path-connected;
		\item[$(\mathrm{RC2})$] For every $U\in \U$, the inclusion
		\[
		\im(\pi_1(U\cap A,x) \rightarrow \pi_1(X,x))\hookrightarrow \im(\pi_1(U,x)\rightarrow \pi_1(X,x))
		\]
		is an isomorphism for every $x \in U\cap A$. 
	\end{enumerate}
	Then the comparison map $\comp^n\colon H^n_b(X,A;\R)\rightarrow H^n(X,A;\R)$ vanishes for every $n \geq \mult(\U)$.
\end{intro_thm}
\begin{rem}
	\label{rem: RC2 is empty}
	In the statement of (RC1) we follow the convention that the empty set is path-connected. Moreover, (RC2) is empty when $\U$ consists of \emph{$\pi_1$-contractible} subsets (for which the map $\pi_1(U,x)\rightarrow \pi_1(X,x)$ is trivial for every $U \in \U$ and every $x \in U$).
\end{rem}
\begin{intro_thm}
	\label{thm: relative vanishing theorem: rel mult and nerves}
	Let $(X,A)$ be a triangulable pair and assume that the kernel of the morphism $\pi_1(A\hookrightarrow X, x)$ is amenable for every $x \in A$.
	Let $\U$ be a locally-finite amenable cover of $X$ by path-connected open subsets satisfying $(\mathrm{RC1})$ and $(\mathrm{RC2})$.
	\begin{enumerate}
		\item[$(1)$] If $\U$ is weakly-convex on $A$, then the comparison map $\comp^n$ vanishes for every $n\geq \mult_A(\U)$.
		\item[$(2)$] If $\U$ is convex, then for every $n \in \N$ there exists a map $\Theta^n$ such that the following diagram commutes
		\[
		% https://tikzcd.yichuanshen.de/#N4Igdg9gJgpgziAXAbVABwnAlgFyxMJZABgBpiBdUkANwEMAbAVxiRAAkA9MAfQCMAFAA1SAQQCUIAL6l0mXPkIoAjOSq1GLNlzDCxkmXOx4CRMsvX1mrRB24CAcgIA6zgKrjSDnqJfvxBrIgGMaKRKoW1FZatjoCAD5Orh7xpIk+fikG6jBQAObwRKAAZgBOEAC2SGQgOBBIqhrWbK4AxpVo3NJBZZXV1HVIAExRmjYgrgAqABYwOHRdhiC9VYgjtfWIAMyjzbZtBHndJeWrOxsNuzF2uq5gTJLUDHR8MAwACvImSiClWHnTHDSChSIA
		\begin{tikzcd}
			{H^n_b(X,A;\R)} \arrow[r, "\comp^n"] \arrow[d, "\Theta^n"] & {H^n(X,A;\R)}                                      \\
			{H^n(N(\U),N(\U_{A});\R)} \arrow[r, "\cong"]                 & {H^n(|N(\U)|,|N(\U_{A})|;\R),} \arrow[u, "H^n(\nu)"']
		\end{tikzcd}
		\]
		where $N(\U)$ is the nerve of $\U$, $N(\U_A)$ is the nerve of the cover of $A$ induced by $\U$ and $\nu\colon (X,A)\rightarrow (|N(\U)|,|N(\U_A)|)$ is a nerve map.
	\end{enumerate}
\end{intro_thm}
We refer the reader to Subsection \ref{subsection: open covers} for the definitions of $\mult_A(\U)$, \emph{weak-convexity} and \emph{convexity}.
Since $\dim(N(\U), N(\U_A)) = \mult_A(\U) - 1$ \cite[Definition 4.4]{li2022bounded}, assuming convexity, we have that (2) implies (1).
\begin{rem}
	\label{rem: connection with li loeh and moraschini}
	Theorem \ref{thm: main theorem} and Theorem \ref{thm: relative vanishing theorem: rel mult and nerves} are natural extensions of the Relative Vanishing Theorem by Li, L{\"o}h and Moraschini \cite[Theorem 1.1]{li2022bounded}. Their result in fact deduces the same conclusions (1) and (2) of our Theorem \ref{thm: relative vanishing theorem: rel mult and nerves}, under the assumptions that $(X,A)$ is a CW-pair such that $A$ is $\pi_1$-injective in $X$ and $\U$ is a (not necessarily locally-finite) boundedly-acyclic open cover by path-connected subsets satisfying (RC1) and (RC2).
	Of course, triangulable pairs are in particular CW-pairs. However, not every CW-pair can be triangulated \cite[Corollary 4.6.12]{FP90}.
	Moreover, amenable open covers are in particular boundedly-acyclic, and it is not clear whether our results can be extended to this case.
	On the other hand, our assumption which requires the kernel of the map $\pi_1(A\hookrightarrow X)$ to be amenable is clearly more flexible than $\pi_1$-injectivity.
	We also underline the fact that, with respect to \cite{li2022bounded}, we need to assume $\U$ to be locally finite in Theorem \ref{thm: relative vanishing theorem: rel mult and nerves}.
\end{rem}

Via a standard duality argument \cite[Section 7.5]{frigerio2017bounded}, Theorem \ref{thm: main theorem} and Theorem \ref{thm: relative vanishing theorem: rel mult and nerves} imply in a straightforward way vanishing results for relative simplicial volume (see for example \cite[Corollary 6.15]{li2022bounded}). As already noticed in \cite[Section 6.4]{li2022bounded}, these results are however strictly weaker than the main available result for relative simplicial volume \cite[Theorem 3.13]{LMR2022simplicial}, which is based on Gromov's vanishing theorem for non-compact manifolds \cite[Corollary 11]{FM23}.
We conclude with a discussion on the optimality of our assumptions.
\begin{rem}
\label{rem: optimality of assumptions}
Let $S$ be an oriented compact connected surface of genus one with one boundary component. 
It is well known that the relative simplicial volume $\|S,\partial S\|$ of $S$ is positive \cite{gromov1982volume}, hence the comparison map $\comp^2\colon H^2_b(S,\partial S; \R)\rightarrow H^2(S, \partial S;\R)$ is non-zero \cite[Section 7.5]{frigerio2017bounded}.
Moreover, the boundary $\partial S$ is $\pi_1$-injective in $S$.
Therefore, the open covers of $S$ in Figure \ref{fig: optimality of assumtpions} show that the regularity assumptions (RC1) and (RC2) of Theorem \ref{thm: main theorem} are optimal. We refer to \cite[Remark 6.16]{li2022bounded} for a discussion on the optimality of the assumptions of Theorem \ref{thm: relative vanishing theorem: rel mult and nerves}.
\end{rem}
\begin{figure}[ht]
\includegraphics[scale=0.5]{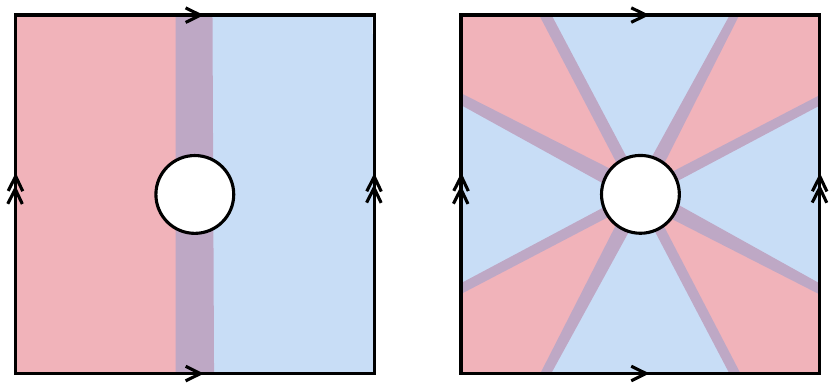}
\caption{These open covers of $S$ are amenable and have multiplicity 2. The one on the left satisfies (RC1), but not (RC2), the one on the right satisfies (RC2), but not (RC1).}
\label{fig: optimality of assumtpions}
\end{figure}
\subsection*{Acknowledgments}
We wish to thank Marco Moraschini for suggesting the topic of this paper. We are grateful to Roberto Frigerio for his guidance. We would also like to thank Federica Bertolotti, Filippo Bianchi, Kevin Li and Francesco Milizia for useful discussions.
\section{Preliminaries}
\label{sec: preliminaries}
\subsection{Open covers}
\label{subsection: open covers}
Let $X$ be a topological space.
An \emph{open cover} $\U$ of $X$ is a set of open subsets of $X$ such that $\bigcup_{U \in \U} U = X$.
An open cover $\U$ of $X$ is called \emph{amenable} if, for every $U \in \U$ and every $x \in U$, we have that $\im(\pi_1(U\hookrightarrow X, x))$ is an amenable subgroup of $\pi_1(X,x)$.
The \emph{multiplicity} of $\U$ is the supremum of the set of natural numbers $k \in \N_{\geq 1}$ such that $U_1\cap \dots \cap U_k \neq \emptyset $ for some pairwise distinct $U_1,\dots,U_k \in \U$.
We say that $\U$ is \emph{convex} if, for every $k \in \N$ and every $U_1,\dots,U_k \in \U$, the intersection $U_1\cap \dots \cap U_k$ is path-connected.
The \emph{nerve} $N(\U)$ of $\U$ is the simplicial complex whose set of vertices is the set $\U$ itself and pairwise-different $U_0,\dots, U_n \in \U$ span an $n$-simplex in $N(\U)$ if $U_0\cap \dots\cap U_n \neq \emptyset$. Of course, $\dim(N(\U))= \mult(\U)-1$.

Let $(X,A)$ be a pair of topological spaces.
The following notions were introduced by Li, L{\"o}h and Moraschini in \cite{li2022bounded}.
An open cover $\U$ of $X$ is called \emph{weakly-convex on $A$} if, for every $k \in \N$ and every $U_1,\dots,U_k \in \U$ with $U_1\cap \dots \cap U_k \cap A \neq \emptyset$, each path-connected component of $U_1\cap \dots \cap U_k$ intersects $A$.
Convex open covers are also weakly convex.
The \emph{relative multiplicity} $\mult_A(\U)$ of $\U$ (with respect to $A$) is defined as the supremum of the set of natural numbers $k \in \N_{\geq 1}$ such that $U_1\cap \dots \cap U_k \neq \emptyset$ and $U_1\cap \dots \cap U_k \cap A  = \emptyset$ for some pairwise distinct $U_1,\dots,U_k \in \U$. Of course, $\mult_A(\U)$ is a lower bound of $\mult(\U)$.

The regularity conditions (RC1) and (RC2) are used in our proof of Theorem \ref{thm: main theorem} via the following lemma.
\begin{lemma}
	\label{lemma: open covers with RC1 and RC2}
	Let $(X,A)$ be a pair of topological spaces and let $\U$ be an cover of $X$ by path-connected open subsets. Then $\U$ satisfies the following conditions:
	\begin{enumerate}
		\item[$(\mathrm{RC1})$] For every $U\in \U$, $U\cap A$ is path-connected; 
		\item[$(\mathrm{RC2})$] For every $U\in \U$, the inclusion
		\[
		\im(\pi_1(U\cap A,x) \rightarrow \pi_1(X,x))\hookrightarrow \im(\pi_1(U,x)\rightarrow \pi_1(X,x))
		\]
		is an isomorphism for every $x \in U\cap A$;
	\end{enumerate}
	if and only if, for every $U \in \U$ and for every path $\gamma\colon [0,1]\rightarrow U$ with endpoints in $U \cap A$, there exists a path $\lambda \colon [0,1]\rightarrow U\cap A$ such that $\lambda$ is homotopic in $X$ to $\gamma$ relative to the endpoints.
\end{lemma}
\begin{proof}
	Assume that $\U$ satisfies (RC1) and (RC2). 
	Given $U\in \U$ and $\gamma$ as above, by (RC1) there exists a path $\varepsilon \colon [0,1]\rightarrow U\cap A$ from $\gamma(0)$ to $\gamma(1)$. The homotopy class of the loop $\gamma * \bar{\varepsilon}$ identifies an element of $\im(\pi_1(U\rightarrow X,\gamma(0)))$, hence, by (RC2), there exists a loop $\delta \colon [0,1] \rightarrow U\cap A$ which is homotopic in $X$ to $\gamma * \bar{\varepsilon}$ relative to the endpoints. Therefore, we can set $\lambda \coloneqq \delta * \varepsilon$.
	The converse implication is straightforward.
\end{proof}

\subsection{Relative bounded cohomology of pairs}
Let $(X,A)$ be a pair of topological spaces. Let $C^\bullet(X)$ and $C^\bullet(A)$ denote the real singular cochain complexes of $X$ and $A$ respectively. We denote by $C^\bullet(X,A)$ the kernel of the restriction map $C^\bullet(X)\rightarrow C^\bullet(A)$. The \emph{real bounded cohomology} $H^\bullet_b(X,A)$ of $(X,A)$ is the cohomology of the subcomplex $C^\bullet_b(X,A)\subseteq C^\bullet(X,A)$ given by bounded cochains, where a cochain $f \in C^n(X)$ is called \emph{bounded} if 
\[
\lVert f \rVert_\infty =
\sup \bigl\{ | f(\sigma) |, \; \sigma \mbox{ is a singular $n$-simplex}
\bigr\} < \infty .
\]
The inclusion of complexes $C^\bullet_b(X,A)\subseteq C^\bullet(X,A)$ induces the so-called \emph{comparison map} $\comp^n\colon H^n_b(X,A)\rightarrow H^n(X,A)$.
\begin{comment}
As for usual singular cohomology, the short exact sequence of complexes 
\[
0\rightarrow C^\bullet_b(X,A)\rightarrow C^\bullet_b(X)\rightarrow C^\bullet_b(A)\rightarrow 0
\]
induces the long exact sequence
\[
\dots \rightarrow H^{n-1}_b(A)\rightarrow H^n_b(X,A)\rightarrow H^n_b(X)\rightarrow H^n_b(A)\rightarrow \dots
\]
\end{comment}

\subsection{Multicomplexes}
\label{sec: multicomplexes}
Multicomplexes are simplicial structures introduced by Gromov in \cite{gromov1982volume}. A \emph{multicomplex} can be defined as a $\Delta$-complex that is both regular and unordered (see \cite{hatcherAT}) i.e. a $\Delta$-complex whose simplices have unordered and distinct vertices, where multiple simplices may share the same set of vertices. 
Every simplicial complex is of course also a multicomplex.
We refer the reader to \cite{FM23} for a more precise discussion of the notion of multicomplex.

A pair of multicomplexes is a pair $(K,L)$, where $L$ is a submulticomplex of $K$.
A simplicial map between pairs of multicomplexes $f\colon (K,L)\rightarrow (K',L')$ is given by a simplicial map $f\colon K\rightarrow K'$ such that $f(L)\subseteq L'$.
We denote by $\Aut(K,L)$ the group of simplicial automorphisms of the pair $(K,L)$. 
Given a group $G$, a simplicial group action $G\acts (K,L)$ is given by a group homomorphism $G\rightarrow \Aut(K,L)$.

Let $K$ be a multicomplex. An \emph{algebraic $n$-simplex} of $K$ is a pair 
\[
\sigma = (\Delta, (v_0,\dots, v_n)),
\]
where $\Delta$ is a $k$-simplex of $K$, and the set $\{v_0,\dots,v_n\}$ is the set of vertices of $\Delta$. We do not require the elements of the ordered $(n+1)$-tuple $(v_0,\dots,v_n)$ to be pairwise distinct. However, since $\Delta$ has exactly $k+1$ vertices, one has $k\leq n$. 
We denote by $C_n(K)$ the vector space generated by algebraic $n$-simplices.
For every $i \in \{0,\dots, n\}$, the $i$-th face $\partial_n^i \sigma$ of $\sigma$ is given by 
\[
\partial_n^i \sigma = (\Delta',(v_0,\dots, \hat{v}_i,\dots, v_n)) \in C_{n-1}(K),
\]
where $\Delta'=\Delta$, if $\{v_0,\dots,v_n\}= \{ v_0,\dots, \hat{v_i}, \dots, v_n\}$, and $\Delta'$ is the unique face of $\Delta$ with vertices $\{ v_0,\dots, \hat{v_i}, \dots, v_n\}$, otherwise.
The boundary operator $\partial_n\colon C_n(K)\rightarrow C_{n-1}(K)$, defined as the alternating sum of the faces $\partial_n^i$, endows $C_\bullet(K)$ with the structure of a chain complex.
Let $C^\bullet(K)$ (resp. $C^\bullet_b(K)$) be the complex of (bounded) real simplicial cochains on $K$, and let $H^\bullet(K)$ (resp. $H^\bullet_b(K)$) be the corresponding cohomology module. 

If $L$ is a submulticomplex of $K$, the inclusion $C_\bullet(L)\subseteq C_\bullet(K)$ induces a short exact sequence of chain complexes
\[
0\rightarrow C^\bullet_b(K,L)\rightarrow C^\bullet_b(K)\rightarrow C^\bullet_b(L)\rightarrow 0.
\]
We denote by $H^\bullet_b(K,L)$ the cohomology of the complex $C^\bullet_b(K,L)$, which in turn fits into the long exact sequence
\[
\dots \rightarrow H^{n-1}_b(L)\rightarrow H^n_b(K,L)\rightarrow H^n_b(K)\rightarrow H^n_b(L)\rightarrow \dots
\]

A cochain $f \in C^n(K)$ is called \emph{alternating} if, for every algebraic simplex $(\Delta, (v_0,\dots, v_n))$ and every permutation $\tau$ of the set $\{0,\dots,n\}$, we have that $f(\Delta, (v_0,\dots, v_n)) = \text{sign}(\tau) \cdot f(\Delta, (v_{\tau(0)},\dots, v_{\tau(n)}))$.
We denote by $C^\bullet_b(K)_\alt$ the subcomplex of alternating cochains. Since the inclusion $C^\bullet_b(K)_\alt \subseteq C^\bullet_b(K)$ induces an isomorphisms in cohomology \cite[Theorem 1.9]{FM23}, it follows from the Five Lemma that the inclusion 
\[
C^\bullet_b(K,L)_\alt=\ker(C^\bullet_b(K)_\alt\rightarrow C^\bullet_b(L)_\alt)\rightarrow C^\bullet_b(K,L)
\]
induces an isomorphisms in cohomology. In particular, every cohomology class in $H^\bullet_b(K,L)$ can be represented by an alternating cocycle.

\subsection{Invariant chains}
If a group $G$ acts simplicially on a multicomplex $K$, then it induces a linear action $G\acts C_\bullet(K)$ on the space of algebraic simplices.
This in turns induces a linear action $G\acts C^\bullet_b(K)$ on the space of bounded cochains. 
We denote by $C^\bullet_b(K)^G$ the space of $G$-invariant bounded cochains.
Let $L$ be a submulticomplex of $K$ and let $H$ be a subgroup of $G$ whose action on $K$ preserves $L$.
Then there is an obvious restriction map \[r_{G,H}^\bullet\colon C^\bullet_b(K)^G\rightarrow C^\bullet_b(L)^H,\]
whose kernel is $C^\bullet_b(K,L)^{G}\coloneqq C^\bullet_b(K)^G\cap C^\bullet_b(K,L)$.
The following regularity condition for the action $G\acts K$ on $L$ was introduced in \cite[Definition 9.1]{Cap24}.
The action \emph{$G\acts K$ has orbits in $L$ induced by $H$} if, for every algebraic simplex $\sigma$ of $L$ and every $g \in G$ such that $g\cdot \sigma$ is an algebraic simplex of $L$, there exists $h \in H$ such that $h\cdot \sigma = g \cdot \sigma$.

If the action $G\acts K$ has orbits in $L$ induced by $H$, it follows that the restriction map $r_{G,H}^\bullet$ is surjective.
In fact, given an $H$-invariant cochain $z \in C^n_b(L)^H$, we can define in the following way a $G$-invariant cochain $z' \in C^n_b(K)^G$ such that $r_{G,H}^n(z')=z$.
For every algebraic $n$-simplex $\sigma$ of $K$, we set $z'(\sigma)= z(\hat{\sigma})$, if $\sigma = g\cdot \hat{\sigma}$ for some $g \in G$ and $\hat{\sigma}\in C_n(L)$, and $z'(\sigma)=0$, otherwise.
Since $z$ is $H$-invariant and the action $G\acts K$ has orbits in $L$ induced by $H$, we have that $z'$ is indeed well-defined.
Moreover it is clear that $z'$ is $G$-invariant and satisfies $r_{G,H}^n(z')=z$.
To summarize, we have proved the following lemma, which allows us, under suitable circumstances, to take $G$-invariant cocycles in every relative coclass of $H^\bullet_b(K,L)$ (see also \cite[Lemma 9.2]{Cap24}).
\begin{lemma}
	\label{lemma: invariant cochains}
	Let $(K,L)$ be a pair of multicomplexes. Let $G\acts K$ be a simplicial action and let $H$ be a subgroup of $G$ which induces an action $H\acts (K,L)$.
	Assume that the orbits of $G$ in $L$ are induced by $H$. Then the rows of the following commutative diagram are both exact
	\[
	% https://tikzcd.yichuanshen.de/#N4Igdg9gJgpgziAXAbVABwnAlgFyxMJZABgBpiBdUkANwEMAbAVxiRGJAF9T1Nd9CKAIzkqtRizYBhAHoAdOQCMmDBjBwB9RQAoA0qQAyAShnAA4qQASnLjxAZseAkQBMo6vWatEIWQuWq6lp6Jma2vI4CRADM7uJe0vJKKmqaOsYyluH2fE6CyAAscZ6SPhzcEfzOKEVCYiXe7NkOVflkdR4SjeV2LXlEIh3xpb5JAanB+sbNuVEobkMNif4pQTq6RjOR1cixi13LyYFp2tOcYjBQAObwRKAAZgBOEAC2SGQgOBBIQhUgT68ftQvkgXH8AW9EG5Pt9ENFwc9IbEYUgCgjAYgAGzA2EAdnRkNxOKQAA4CaTiYgAJydBI+R4acxWGzk6mUgCsrORIOprOhPLJdghQJRiHxFE4QA
	\begin{tikzcd}
		0 \arrow[r] & {C^\bullet_b(K,L)^{G}} \arrow[r] \arrow[d] & C^\bullet_b(K)^G \arrow[r, "{r_{G,H}^\bullet}"] \arrow[d]  & C^\bullet_b(L)^H \arrow[r] \arrow[d] & 0 \\
		0 \arrow[r] & {C^\bullet_b(K,L)} \arrow[r]                 & C^\bullet_b(K) \arrow[r] & C^\bullet_b(L) \arrow[r]             & 0,
	\end{tikzcd}
	\]
	where the vertical arrows are inclusions of cochains. In particular, if the inclusions $C^\bullet_b(K)^G\hookrightarrow C^\bullet_b(K)$ and $C^\bullet_b(L)^H\hookrightarrow C^\bullet_b(L)$ induce isomorphisms in cohomology, then $C^\bullet_b(K,L)^{G}\hookrightarrow C^\bullet_b(K,L)$ induces an isomorphism in cohomology.
\end{lemma}

\subsection{The singular multicomplex}
\label{subsec: singular multicomplex}
Let $X$ be a topological space. The \emph{singular multicomplex} $\K(X)$ of $X$ is the multicomplex whose simplices are the singular simplices in $X$ with distinct vertices, up to affine parametrization \cite{gromov1982volume}. 
Therefore, the geometric realization $|\K(X)|$ of $\K(X)$ is a CW-complex with 0-skeleton corresponding to the set $X$ itself. Moreover, there is an obvious projection $S_X\colon |\K(X)|\rightarrow X$, which is a homotopy equivalence when $X$ is a CW-complex \cite[Corollary 2.2]{FM23}. 
The size of $\K(X)$ can be reduced without changing its homotopy type. 
We define a submulticomplex $\Elle(X)$ of $\K(X)$ as follows: the 0-skeleton of $\Elle(X)$ is the same as the one of $\K(X)$ (hence it corresponds to the set $X$); having already defined the $n$-skeleton $\Elle(X)^n$ of $\Elle(X)$, we define the $(n+1)$-skeleton by adding to $\Elle(X)^n$ one $(n+1)$-simplex for each homotopy class of $(n+1)$-simplices of $\K(X)$ whose facets are all contained in $\Elle(X)^n$. 
It turns out that the inclusion $|\Elle(X)|\hookrightarrow|\K(X)|$ is a homotopy equivalence \cite[Theorem 3.23]{FM23}.
Even if the construction of $\Elle(X)$ depends on several choices, the submulticomplex $\Elle(X)$ is unique up to simplicial isomorphism \cite[Theorem 3.23]{FM23}.
To the singular multicomplex we can associate an \emph{aspherical multicomplex} $\A(X)$, which is defined as the quotient of $\Elle(X)$, where two simplices of $\Elle(X)$ are identified if and only if they share the same 1-skeleton.
One can define a simplicial projection $\pi\colon \Elle(X)\rightarrow \A(X)$, which restricts to the identity on $\Elle(X)^1=\A(X)^1$. 
This quotient map induces an isomorphism on fundamental groups \cite[Proposition 3.33]{FM23} and kills higher homotopy \cite[Theorem 3.31]{FM23}. 
Therefore the topological realization $|\A(X)|$ of $\A(X)$ is a model for the classifying space of the fundamental group of $X$.

\subsection{The group $\Pi(X,X)$ and its action on $\A(X)$}

Let $X$ be a topological space. 
The group $\Pi(X,X)$, first introduced by Gromov in \cite{gromov1982volume}, is defined as follows.
\begin{defn}
	\label{defn: group Pi(X,X)}
	Let $X_0$ be a subset of $X$. The set $\Omega(X,X_0)$ consists of families of paths $\{\gamma_x\}_{x \in X_0}$ such that:
	\begin{itemize}
		\item[(1)] for every $x \in X_0$, $\gamma_x\colon [0,1]\rightarrow X$ is a continuous path such that $\gamma_x(0)=x$ and $\gamma_x(1) \in X_0$;
		\item[(2)] $\gamma_x$ is constant for all but finitely many $x \in X_0$;
		\item[(3)] the map $X_0\rightarrow X_0$, $x\mapsto\gamma_x(1)$, is a bijection (with finite support).
	\end{itemize}
	The following concatenation of paths endows $\Omega(X,X_0)$ with the structure of a semigroup: given two elements $\{\gamma_x\}_{x \in X_0}$ and $\{\gamma'_x\}_{x \in X_0}$ of $\Omega(X,X_0)$, their concatenation is $\{\gamma_x * \gamma'_{\gamma_x(1)} \}_{x \in X_0}$, where $*$ denotes the usual concatenation of paths.
	In order to obtain a group, we consider the set $\Pi(X,X_0)$ of homotopy classes of elements of $\Omega(X,X_0)$, where two elements $\{\gamma_x\}_{x \in X_0}$ and $\{\gamma'_x\}_{x \in X_0}$ of $\Omega(X, X_0)$ are \emph{homotopic} if $\gamma_x$ is homotopic to $\gamma_x'$ in $X$ relative to the endpoints, for every $x \in X_0$.
\end{defn}

Elements of $\Pi(X,X)$ are usually denoted by listing the homotopically non-trivial paths in one of its representatives.
If $V \subseteq U \subseteq X$, then the inclusion $(U,V)\hookrightarrow (X,X)$ induces a group homomorphism $\Pi(U,V)\rightarrow \Pi(X,X)$, whose image is denoted by $\Pi_X(U,V)$.
If $X_0=\{x_0\}$, then $\Pi(X,X_0)=\pi_1(X,x_0)$ is just the fundamental group of $X$ at $x_0$. In general, we have an injective group homomorphism
\[
\bigoplus_{x \in X_0} \pi_1(X, x)\hookrightarrow \Pi(X,X_0),
\]
which is the kernel of the natural homomorphism from $\Pi(X,X_0)$ to the group of permutations of $X_0$ with finite support \cite[Proposition 6.5]{FM23}. Since the latter is amenable, if every connected component of $X$ has amenable fundamental group, then $\Pi(X,X)$ is amenable. In general, we have the following.
\begin{lemma}[{\cite[Lemma 6.6]{FM23}}]
	\label{lemma: Pi_X(U,V) is amenable}
	Let $U$ be an amenable subset of $X$ and let $V \subseteq U$ be any subset. Then the subgroup $\Pi_X(U,V)\leq \Pi(X,X)$ is amenable.
\end{lemma}

The group $\Pi(X,X)$ acts on $\A(X)$ as follows. 
Let $g\in \Pi(X,X)$ and let $\{\gamma_x\}_{x \in X}$ be a representative of $g$.
We define the action of $g$ on the 0-skeleton of $\A(X)$ as the permutation induced by $g$ on $\A(X)^0=X$.
Let $e$ be an edge of $\A(X)$ with endpoints $v_0,v_1 \in \A(X)^0=X$. Recall that edges of $\A(X)$ (hence of $\Elle(X)$) correspond to homotopy classes (relative to endpoints) of paths in $X$. Let $\gamma_e\colon [0,1]\rightarrow X$ be a representative of $e$. We define $g\cdot e$ as the homotopy class (relative to the endpoints) of the path $\bar{\gamma}_{v_0}* \gamma_e * \gamma_{v_1}$.
Since $\A(X)$ is aspherical, we can uniquely extend $g$ to a simplicial automorphisms of the whole $\A(X)$, which is simplicially homotopic to the identity \cite[Theorem 5.3]{FM23}.

\section{From pairs of spaces to pairs of multicomplexes}
\label{sec: bounded cohomology of pairs}
Let $(X,A)$ be a CW-pair.
Being injective, the inclusion map $j\colon A \hookrightarrow X$ induces a simplicial embedding $j_\K\colon \K(A)\hookrightarrow \K(X)$ at the level of singular multicomplexes.
We denote by $S_X\colon |\K(X)|\rightarrow X$ and $S_A\colon |\K(A)|\rightarrow A$ the corresponding projections. Of course we have that $S_X \circ |j_\K|=j\circ S_A $.
We consider the multicomplexes $\Elle(A)$ and $\Elle(X)$, respectively, and we denote by $i_A\colon \Elle(A)\hookrightarrow \K(A)$ and $i_X\colon \Elle(X)\hookrightarrow \K(X)$ the corresponding simplicial inclusions.
There is a simplicial map $j_\Elle\colon \Elle(A)\rightarrow \Elle(X)$ which sends each simplex $\Delta$ of $\Elle(A)$ to the unique simplex of $\K(X)$ which is homotopic to $\Delta$ relative to the 0-skeleton \cite[Proposition 4.1]{Cap24}.
Moreover, we can construct $\Elle(X)$ and $\Elle(A)$ such that $j_\K \circ i_A = i_X \circ j_\Elle$.
\begin{rem}
	\label{rem: the mapping theorem is not canonical}
	Although the multicomplexes $\Elle(A)$ and $\Elle(X)$ are unique up to simplicial isomorphism \cite[Theorem 3.23]{FM23}, some choices are needed in order to obtain that $j_\K \circ i_A = i_X \circ j_\Elle$. Namely, once $\Elle(A)$ has been constructed, we need to construct $\Elle(X)$ in the following way: if the facets of a simplex are all contained in $A$ and if there is a simplex of $\Elle(A)$ in the same homotopy class, we choose this as a representative. In this paper, whenever we construct a pair of multicomplexes from a CW-pair, we always adhere to this setting.
\end{rem}
The map $j_\Elle$ is not injective in general: to this end one needs to have more control over the homotopy of the pair $(X,A)$ \cite[Proposition 4.2]{Cap24}.
Let $\pi \colon \Elle(X)\rightarrow \A(X)$ denote the simplicial projection which identifies simplices sharing the same 1-skeleton. 
Of course, $j_\Elle$ factors to a well-defined simplicial map $j_\A\colon \A(A)\rightarrow \A(X)$.
We denote by $\A_X(A)$ the image of $\A(A)$ in $\A(X)$ via $j_\A$, so that the pair of multicomplexes $(\A(X),\A_X(A))$ is well-defined. 
We denote by $q\colon \A(A)\rightarrow \A_X(A)$ the surjective map induced by $j_\A$.
If $A$ is $\pi_1$-injective in $X$, then $j_\A$ is a simplicial embedding and $\A_X(A)$ is simplicially isomorphic to $\A(A)$ \cite[Section 1.3]{kuessner2015multicomplexes} \cite[Proposition 4.3]{Cap24}.
In short, we have the following commutative diagram of simplicial maps
\begin{equation}
	\label{eq: diagram simplicial maps for pairs}
	% https://tikzcd.yichuanshen.de/#N4Igdg9gJgpgziAXAbVABwnAlgFyxMJZABgBpiBdUkANwEMAbAVxiRAB12BpACgEEAlCAC+pdJlz5CKAIzkqtRizacAogwYx+Q0eOx4CRMjIX1mrRB248AGjrEgM+qUTknqZ5ZbUatdkQ5OkoYoAEzyHkoWVnzaAXrB0sjh7ormKuyx-rqOEgZJAMwRaV4xAPo2cTlB+URFqZ7RnFk6CjBQAObwRKAAZgBOEAC2SHIgOBBIZON0WAxsABYQEADWjDggkemWWGV88SADw0hF45OI4TNzi8trDBtbpbs2B0cjiNMTSJeNbABWZU4XE2V3mliWq1eg3eYy+iFOv0sAJ8mihx0QsPOABZHk12GgsGj3qc4QBWXEZAlEpA4s5IcklaLIzIgnCzMHgAisHJvGnUOEANgpljQ1MQDLhAHZhVYsFAxUK6YhpaCbpDhBRhEA
	\begin{tikzcd}
		\K(A) \arrow[d, "j_\K", hook] & \Elle(A) \arrow[l, "i_A", hook'] \arrow[d, "j_\Elle"] \arrow[r, "\pi", two heads] & \A(A) \arrow[d, "j_\A"] \arrow[r, "q", two heads] & \A_X(A) \arrow[d, hook] \\
		\K(X)                         & \Elle(X) \arrow[l, "i_X", hook'] \arrow[r, "\pi", two heads]                      & \A(X) \arrow[r, equal]                 & \A(X).                  
	\end{tikzcd}
\end{equation}
\begin{prop}[{\cite[Theorem 7]{Cap24}}]
	\label{prop: relative mapping theorem}
	Let $(X,A)$ be a CW-pair such that the kernel of the morphism $\pi_1(A\hookrightarrow X, x)$ is amenable for every $x \in A$.
	Then, for every $n \in \N$, there is an isomorphism of vector spaces
	\[
	\Psi^n\colon H^n_b(\A(X),\A_X(A))\rightarrow H^n_b(X,A).
	\]
\end{prop}
The previous result is achieved in \cite{Cap24} by using the machinery of mapping cones developed in \cite{Park03, loh2008isomorphisms}.
The map $\Psi^n$ is induced by the maps in (\ref{eq: diagram simplicial maps for pairs}), hence it is \emph{a priori not canonical}, since our construction depends on the choices described in Remark \ref{rem: the mapping theorem is not canonical}.
We refer to the next pages for the details of the construction, which can be also found in \cite[Section 8]{Cap24}.

\subsection{Mapping cones}
Let $f\colon L\rightarrow K$ be a simplicial map between multicomplexes. The \emph{mapping cone complex} of $f$ is $(C^\bullet_b(f\colon L\rightarrow K), \bar{\delta}^\bullet)$, where
\[
C^n_b(f\colon L\rightarrow K)=C^n_b(K)\oplus C^{n-1}_b(L), \quad \bar{\delta}^n(u,v)=(\delta^n(u), - f^n(u)-\delta^{n-1}(v)).
\]
The proof of the following is analogous to the one of \cite[Theorem 3.19]{Park03}.
\begin{lemma}[{\cite[Lemma 7.4]{Cap24}}]
	\label{lemma: mapping cones from pairs}
	Let $(K,L)$ be a pair of multicomplexes and let $i\colon L\hookrightarrow K$ denote the inclusion map. The chain map $\beta^\bullet \colon C^\bullet_b(K,L)\rightarrow C^\bullet_b(i\colon L\hookrightarrow K)$, defined by $\beta^n(u)=(u,0)$, induces an isomorphism in cohomology in every degree.
\end{lemma}
The following lemma is a consequence of the naturality of the short exact sequence of chain complexes $0\rightarrow C^{\bullet-1}_b(L)\rightarrow C^\bullet_b(f\colon L\rightarrow K)\rightarrow C^\bullet_b(K)\rightarrow 0$.
\begin{lemma}
	\label{lemma: naturality of mapping cones}
	For every commutative diagram of simplicial maps
	\[
	% https://tikzcd.yichuanshen.de/#N4Igdg9gJgpgziAXAbVABwnAlgFyxMJZABgBpiBdUkANwEMAbAVxiRABkQBfU9TXfIRQBGclVqMWbdgHJuvEBmx4CRMsPH1mrRCADS8vssFFRG6lqm69cruJhQA5vCKgAZgCcIAWyRkQOBBIAEwWkjogboaRXr6IogFBiADMYdpsbrYKnj5+1IFICZYRADolAB5YAPqcPO6xIflJqRLpumWVVQZ2XEA
	\begin{tikzcd}
		L \arrow[d, "f"] \arrow[r, "\xi_L"] & L' \arrow[d, "f'"] \\
		K \arrow[r, "\xi_K"]                & K'                
	\end{tikzcd}
	\]
	there is a chain map $\xi^\bullet\colon C^\bullet_b(f'\colon L'\rightarrow K')\rightarrow C^\bullet_b(f\colon L\rightarrow K)$, defined by $\xi^n(u,v)=(\xi_K^n(u),\xi_L^{n-1}(v))$. Moreover, if both $\xi_K$ and $\xi_L$ induce isomorphisms in bounded cohomology, then also $\xi^\bullet$ does.
\end{lemma}

\subsection{Proof of Proposition \ref{prop: relative mapping theorem}}
Let $(X,A)$ be a CW-pair and consider the commutative diagram of simplicial maps (\ref{eq: diagram simplicial maps for pairs}). We know that $i_X$ and $i_A$ induce isomorphisms in bounded cohomology \cite[Corollary 4.9]{FM23}. The same holds for the map $\pi$ \cite[Theorem 4.23]{FM23}. Assume that the kernel of the morphism $\pi_1(A\hookrightarrow X,x)$ is amenable for every $x \in A$. This last assumption is crucial to show that also the map $q$ induces an isomorphism in bounded cohomology \cite[Lemma 8.3]{Cap24}. It follows that every horizontal map of diagram (\ref{eq: diagram simplicial maps for pairs}) induces an isomorphism in bounded cohomology.

In the following, we implicitly invoke the Five Lemma several times to pass from the absolute case to the relative one. We know that the projection $S\colon (|\K(X)|,|\K(A)|)\rightarrow (X,A)$ is a homotopy equivalence both on $|\K(X)|$ and $|\K(A)|$ \cite[Corollary 2.2]{FM23}. If follows that the chain map $S^\bullet\colon C^\bullet_b(X,A)\rightarrow C^\bullet_b(|\K(X)|,|\K(A)|)$ induces an isomorphism in bounded cohomology.
Let $\varphi_\bullet\colon C_\bullet(\K(X))\rightarrow C_\bullet(|\K(X)|)$ denote the natural chain inclusion.
By the Isometry Lemma \cite[Theorem 2]{FM23}, $\varphi_\bullet$ induces an isomorphism in bounded cohomology. Hence the map $\varphi^\bullet\colon C^\bullet_b(|\K(X)|,|\K(A)|)\rightarrow C^\bullet_b(\K(X),\K(A))$ induced by $\varphi_\bullet$ also does.
By Lemma \ref{lemma: mapping cones from pairs}, the chain map $\beta^\bullet\colon C^\bullet_b(\K(X),\K(A))\rightarrow C^\bullet_b(j_\K\colon \K(A)\hookrightarrow \K(X))$ induces an isomorphism in cohomology.
Since every horizontal map of diagram (\ref{eq: diagram simplicial maps for pairs}) induces an isomorphism in bounded cohomology, we can apply Lemma \ref{lemma: naturality of mapping cones} to every square of diagram (\ref{eq: diagram simplicial maps for pairs}). It follows that the chain maps
\begin{align*}
	i^\bullet \colon C^\bullet_b(j_\K\colon \K(A)\hookrightarrow \K(X)) &\rightarrow  C^\bullet_b(j_\Elle\colon \Elle(A)\rightarrow \Elle(X)), \\
	\pi^\bullet \colon C^\bullet_b(j_\A \colon \A(A)\rightarrow \A(X)) &\rightarrow C^\bullet_b(j_\Elle\colon \Elle(A)\rightarrow \Elle(X)), \\
	q^\bullet \colon C^\bullet_b(\A_X(A)\hookrightarrow \A(X))  & \rightarrow C^\bullet_b(j_\A \colon \A(A)\rightarrow \A(X))
\end{align*}
induce isomorphisms in cohomology. Since $\A_X(A)$ is a subcomplex of $\A(X)$, we can apply Lemma \ref{lemma: mapping cones from pairs} to get a chain map $\beta^\bullet\colon C^\bullet_b(\A(X),\A_X(A))\rightarrow C^\bullet_b(\A_X(A)\hookrightarrow \A(X))$, which induces an isomorphism in cohomology.

In conclusion, both the maps $\zeta^\bullet = i^\bullet\circ \beta^\bullet \circ \varphi^\bullet\circ S^\bullet$ and $\vartheta^\bullet = \pi^\bullet\circ q^\bullet\circ \beta^\bullet$ induce isomorphisms in cohomology. For every $n \in \N$, the map $\Psi^n$ is then defined as the composition $H^n(\zeta^\bullet)^{-1}\circ H^n(\vartheta^\bullet)$.

\section{Proof of Theorem \ref{thm: main theorem}}
\label{sec: proof of main theorem}
Let $(X,A)$ be a triangulable pair such that the kernel of the morphism $\pi_1(A\hookrightarrow X, x)$ is amenable for every $x \in A$.
Let $(T,S)$ be a pair of simplicial complexes such that $(X,A)=(|T|,|S|)$.

There is a natural way to construct a copy of $T$ inside $\A(X)$ \cite[Section 6.2]{FM23}.
Inside the singular multicomplex $\K(X)$ we can find a submulticomplex $\K_T(X)\cong T$ whose simplices correspond to equivalence classes of affine parametrizations of simplices of $T$. 
Moreover, we can assume that $\K_T(X)\subseteq \Elle(X)$, by choosing simplices of $\K_T(X)$ as representatives in their homotopy class.
Since $T$ is a simplicial complex, every simplex in $\K_T(X)$ is uniquely determined by its 0-skeleton, hence the quotient map $\pi\colon \Elle(X)\rightarrow \A(X)$ is injective on $\K_T(X)$. Therefore we can construct a copy of $T$ inside $\A(X)$.
The same construction applies to $S$, so that $\K_S(A)\subseteq \Elle(A)$. As before, both $\pi\colon \Elle(A)\rightarrow \A(A)$ and $q\colon \A(A)\rightarrow \A_X(A)$ are injective on $\K_S(A)$. It follows that $S$ sits inside $\A_X(A)\cap T$.
Moreover, the choices which ensure that $\K_T(X)\subseteq \Elle(X)$ and $\K_S(A)\subseteq \Elle(A)$ are clearly compatible with the ones described in Remark \ref{rem: the mapping theorem is not canonical}, so that both the squares of the following diagram are commutative
\begin{equation}
	\label{eq: restriction mapping cones}
	% https://tikzcd.yichuanshen.de/#N4Igdg9gJgpgziAXAbVABwnAlgFyxMJZABgBpiBdUkANwEMAbAVxiRAB12BpAfQGUAFAEEAlCAC+pdJlz5CKAIzkqtRizacAogwYxhYydOx4CRMgpX1mrRB248AKgIAaBqSAzG5RJRepX1Wy0dPVcJd09ZUxQAJmV-NRs7Ln1woyj5ZDi-VWsNbhcDFRgoAHN4IlAAMwAnCABbJDIQHAgkOJa6LAY2AAsICABrNJBahqbqVqQlTu6+geHDUbrGxA6pxABmSa6e236hkbHVmY3t3MCQACseYN0jlenJtsQAFh25-YWQBLzbLB4Qge4y2zyQAFZfpcAUCPnsQAdFu5jkh3i0XpDZvDET8LkkbpwuBIKOIgA
	\begin{tikzcd}
		\K_S(A) \arrow[d, hook] \arrow[r, hook] & \Elle(A) \arrow[d, "j_\Elle"] \arrow[r, "i_A", hook] & \K(A) \arrow[d, "j_\K", hook] \\
		\K_T(X) \arrow[r, hook]                 & \Elle(X) \arrow[r, "i_A", hook]                      & \K(X).                        
	\end{tikzcd}
\end{equation}
We can therefore invoke Proposition \ref{prop: relative mapping theorem} to get an isomorphism of vector spaces $\Psi^n \colon H^n_b(\A(X), \A_X(A))\rightarrow H^n_b(X,A)$. To this end, it is crucial that the kernel of the map $\pi_1(A\hookrightarrow X)$ is amenable. This is the only step in our proof in which we use this assumption.
\begin{lemma}
	\label{lemma: vanishing on simplicial complex and comparison map}
	Let $z \in C^n_b(\A(X),\A_X(A))$ be a bounded cocycle which vanishes on $C_n(T)\subseteq C_n(\A(X))$. Then $\comp^n\circ \Psi^n ([z])=0$ in $H^n(X,A)$.
\end{lemma}
\begin{proof}
	We treat simplicial complexes as specific cases of multicomplexes, using the cohomology theory defined in Section \ref{sec: multicomplexes}.
	The inclusion of pairs $(T,S)\subseteq (\A(X),\A_X(A))$ induces a restriction map $r^n_\A \colon C^n_b(\A(X),\A_X(A))\rightarrow C^n_b(T,S)$. 
	Similarly, by Lemma \ref{lemma: mapping cones from pairs}, the leftmost square of diagram (\ref{eq: restriction mapping cones}) induces a restriction map $r^n_\Elle \colon C^n_b(\Elle(A)\rightarrow \Elle(X))\rightarrow C^n_b(S\hookrightarrow T)$ on mapping cones.
	We consider the following diagram
	\[
	% https://tikzcd.yichuanshen.de/#N4Igdg9gJgpgziAXAbVABwnAlgFyxMJZABgBpiBdUkANwEMAbAVxiRAGEA9MAfQCMAFAB0hAQQEANAJSkRonhIGipUkAF9S6TLnyEUZAIxVajFmy69BAKx4iAogwYwRAYwgMCAAnuOYSqSIATlgA5gAWOHSBgRAA7t5CDk6SKuqaIBjYeAREZABMxvTMrIgc3PySpMppWlm6RAbkhaYlZZYCACqkAMqqGrU6OSiNRtRFZqUWFd0iYRAQANbB4ZHRcZ4dfemZg3rIeU1jLebcnT1bA9l7B6MmxSdgAt0Jc4vLEVEx8Zs1GdpXRAAzIc7hM2mder8dgCUMCCkd7pNThIqn1jDAoCF4ERQAAzGIAWyQwJAOAgSAArNRIlgGGxXgtfviIETEAAWankxAANmpdFp9PmjP6IGZrKppK5vNBrREfBgkW4TMJxM5SA5MrYcoVdCVIrFlLViAA7NQGFgwK0oBAmHwnCBqGEYHQoGxIJaHaT+XTSu7WPqVYgDpKkAAOPkC0oM5UssNG02a0oieiBNBhLBKs10eUMAAK-3qpScuJwMdZZBDiBJ41agT16QNiEalY1NbYdcIAdjiArZKQzbbSaEAC8dfW8YHg32mwiwcmojgnYrCFmc-m6kMQMXS2oKGogA
	\begin{tikzcd}
		{C^n_b(\A(X),\A_X(A))} \arrow[r, "r^n_\A"] \arrow[d, "\zeta^n"]       & {C^n_b(T,S)} \arrow[r, hook] \arrow[d, "\beta^n"] & {C^n(T,S)} \arrow[d, "\beta^n"] \arrow[r, equal] & {C^n(T,S)}                         \\
		C^n_b(\Elle(A)\rightarrow \Elle(X)) \arrow[r, "r^n_\Elle"] & C^n_b(S\hookrightarrow T) \arrow[r, hook]         & C^n(S \hookrightarrow T)                                       &                                    \\
		{C^n_b(X,A)} \arrow[rrr, hook] \arrow[u, "\vartheta^n"]            &                                                   &                                                                & {C^n(X,A).} \arrow[uu, "\varphi^n"]
	\end{tikzcd}
	\]
	Here hooked arrows denote inclusion of chains, which induce comparison maps in cohomology. The maps $\zeta^n$ and $\vartheta^n$, introduced in Section \ref{sec: bounded cohomology of pairs}, induce isomorphisms in cohomology whose composition defines $\Psi^n$. The maps $\beta^n$, defined in Lemma \ref{lemma: mapping cones from pairs}, and the map $\varphi^n$, induced by the chain inclusion $C_n(T)\rightarrow C_n(|T|)=C_n(X)$, also induce isomorphisms in cohomology, by Lemma \ref{lemma: mapping cones from pairs} and \cite[Theorem 1.11]{FM23} respectively.
	Since $z$ vanishes on $C_n(T)$, we have that $r^n_\A(z)=0$.
	Since the diagram is commutative and every vertical arrow induces isomorphism in cohomology, the statement easily follows.
\end{proof}
Let $\U=\{U_i \,| \, i \in I\}$ be an amenable open cover of $X$ by path-connected subsets satisfying conditions (RC1) and (RC2).
We denote by $V$ the set of vertices of $T$.
For every vertex $v \in V$, the \emph{closed star} of $v$ in $T$ is the subcomplex of $T$ containing all the simplices containing $v$.
By suitably subdividing $T$, we can assume that for every $v \in V$ there exists $i(v)\in I$ such that the closed star of $v$ in $T$ is contained in $U_{i(v)}$ \cite[Theorem 16.4]{munkres1984elements}. Of course the choice of $i(v)$ may not be unique.
For every $i \in I$, we set $V_i = \{v \in V(T)\,|\; i(v)=i\}$.
We consider the group
\[
G = \bigoplus_{i \in I} \Pi_X(U_i,V_i).
\]
Since $U_i$ is amenable in $X$, then $\Pi_X(U_i,V_i)$ is an amenable group by Lemma \ref{lemma: Pi_X(U,V) is amenable}. 
The direct sum of amenable groups is amenable, hence also $G$ is amenable. 
Moreover, since $V_i\cap V_j =\emptyset$ for every $i\neq j$, elements of $\Pi_X(U_i, V_i)$ commute with elements of $\Pi_X(U_j,V_j)$.
It follows that there is a well-defined map from $G$ to $\Pi(X,X)$.
This map is injective because the intersection $\Pi_X(U_i,V_i)\cap \left\langle \Pi_X(U_j,V_j) \colon  j \neq i \right\rangle$ is trivial for every $i \in I$. In fact, non-trivial elements of the former group are represented by paths with endpoints in $V_i$, while non-trivial elements of the latter by paths with endpoints in $\bigcup_{j\neq i} V_i$, which is disjoint from $V_i$.
Therefore $G$ can be identified with a subgroup of $\Pi(X,X)$ and acts on $\A(X)$ accordingly. Notice that this action does not preserve the submulticomplex $\A_X(A)$.
The group
\[
H = \bigoplus_{i \in I} \Pi_X(U_i\cap A,V_i\cap A)
\]
identifies a subgroup of $G$ and induces an action on the pair $(\A(X), \A_X(A))$. Being a subgroup of an amenable group, $H$ is amenable.
The following lemma exploits the regularity conditions (RC1) and (RC2) satisfied by the open cover.

\begin{lemma}
	\label{lem: G has orbits induced by H}
	The action $G\acts \A(X)$ has orbits in $\A_X(A)$ induced by $H$.
\end{lemma}
\begin{proof}
	Let $\sigma = (\Delta,(x_0,\dots, x_n))$ be an algebraic $n$-simplex of $\A_X(A)$ and let $g \in G$ be such that $g\cdot \sigma$ is an algebraic $n$-simplex of $\A_X(A)$. We need to show that there exists an element $h \in H$ such that $h\cdot \sigma = g\cdot \sigma$.
	Let $\{v_0,\dots, v_k\}$ be the vertices of $\Delta$, where $k\leq n$, and let $\{\gamma_x\}_{x \in X}$ be any representative of $g$.
	Since $v_j$ and $g\cdot v_j$ are both vertices of $\A_X(A)$ (hence points of $A$), by the definition of the action $G \acts \A(X)$, we have that $\gamma_{v_j}$ has both endpoints in $A$.
	Moreover, by the definition of $G$, we have that $\gamma_{v_j}$ is supported in some element of $\U$. 
	Hence there are $i_0,\dots,i_k \in I$ such that $\gamma_{v_j}\colon [0,1]\rightarrow U_{i_j}$ has both endpoints in $U_{i_j}\cap A$.
	Since the open cover $\U$ satisfies (RC1) and (RC2), by Lemma \ref{lemma: open covers with RC1 and RC2} there exist paths $\lambda_{v_j}\colon [0,1] \rightarrow U_{i_j}\cap A$, $j \in \{0,\dots, k\}$, such that $\lambda_{v_j}$ is homotopic to $\gamma_{v_j}$ in $X$ relative to the endpoints.
	Therefore, if we set $h_j = \{\lambda_j\}$ if $\lambda_j$ is a loop, and $h_j=\{\lambda_j,\bar{\lambda}_j\}$ otherwise, then $h=\oplus_{j=0}^k h_j$ defines an element of $H$. 
	It is easy to check that $h\cdot e = g\cdot e$, for every edge $e$ of $\Delta$, hence $g\cdot \Delta^1=h\cdot \Delta^1$, where $\Delta^1$ denotes the 1-skeleton of $\Delta$. Since $\A(X)$ is an aspherical multicomplex, this is indeed sufficient to conclude \cite[Proposition 3.30]{FM23}.
\end{proof}

Let $n \geq \mult(\U)$ and let $\alpha \in H^n_b(X,A)$.
Since  $\Psi^n$ is an isomorphism, we can take a bounded coclass $\beta \in H^n_b(\A(X),\A_X(A))$ such that $\Psi^n(\beta)=\alpha$.
Let now $z \in C^n_b(\A(X),\A_X(A))$ be an alternating cocycle representing $\beta$.
The actions $G\acts \A(X)$ and $H\acts \A_X(A)$ are both by amenable groups and by automorphisms which are simplicially homotopic to the identity \cite[Theorem 5.3]{FM23} \cite[Proposition 6.1]{Cap24}.
It follows from these two facts that the inclusion maps $C^\bullet_b(\A(X))^G\hookrightarrow C^\bullet_b(\A(X))$ and $C^\bullet_b(\A_X(A))^H\hookrightarrow C^\bullet_b(\A_X(A))$ induce isomorphisms in cohomology \cite[Theorem 4.21]{FM23}. 
Moreover, since the action $G\acts \A(X)$ has orbits in $\A_X(A)$ induced by $H$, by Lemma \ref{lemma: invariant cochains} we can assume that $z$ is alternating and $G$-invariant.
By Lemma \ref{lemma: vanishing on simplicial complex and comparison map}, it suffices to show that $z$ vanishes on every algebraic $n$-simplex of $C_n(T)\subseteq C_n(\A(X))$.

Let $(\Delta,(x_0,\dots,x_n))$ be an algebraic $n$-simplex in $C_n(T)\subseteq C_n(\A(X))$.
If $x_h=x_k$, for some $h\neq k$, then $z(\Delta,(x_0,\dots,x_n))=0$, since $z$ is alternating.
We assume therefore that the points $x_0,\dots,x_n$ are pairwise distinct.
By construction, the cover of $X$ by the closed star of vertices of $T$ is a refinement of $\U$. 
Since $\Delta$ is an $n$-simplex of $T$ and since $n\geq \mult(\U)$, there are at least two vertices of $\Delta$ belonging to the same $V_i$. In other words, there exist $h\neq k \in \{0,\dots, n\}$ such that $i(x_h)=i(x_k)$.
Let $e$ denote the edge of $\Delta$ joining $x_h$ with $x_k$.
By assumption, the closed stars of $x_k$ and $x_h$ are both contained in $U_i$, therefore the edge $e$ of $\A(X)$ (which is also an edge of $\Elle(X)^1=\A(X)^1$, hence of $\K(X)$) projects via $S_X\colon |\K(X)|\rightarrow X$ to a path $\gamma \colon [0,1]\rightarrow U_i$ with endpoints $x_h$ and $x_k$. 
As a consequence, if we consider $g=\{\gamma,\bar{\gamma}\}\in \Pi_X(U_i,V_i)<G$, it is easy to check that $g\cdot \Delta = \Delta$, $g\cdot x_h = x_k$, $g\cdot x_k = x_h$ and $g\cdot x_j= x_j$, for every $j \neq h,k$. 
Since $z$ is $G$-invariant, we obtain $z(\Delta, (x_0,\dots, x_n))= z(g\cdot (\Delta, (x_0,\dots, x_n)))$, while, since $z$ is alternating, we have $z(\Delta, (x_0,\dots, x_n))= - z(g\cdot (\Delta, (x_0,\dots, x_n)))$.
Therefore $z(s,(x_0,\dots,x_n))=0$ and this concludes the proof.
\begin{rem}
	\label{rem: role of T and S}
	In our proof of Theorem \ref{thm: main theorem} we work with a pair of \emph{simplicial complexes} $(T,S)$ for two main technical reasons. The first is to exploit the barycentric techniques to relate stars of $T$ with the open cover $\U$. The second is to build a copy of $(X,A)$ inside $(\A(X),\A_X(A))$, on which $G$ acts. 
	%Of course, our argument works as well if $(T,S)$ is a pair of \emph{multicomplexes}. However, this fact does not add generality to Theorem \ref{thm: main theorem}, since the second barycentric subdivision of every multicomplex is a simplicial complex.
\end{rem}
\begin{rem}
	\label{rem: relation with Kuessner paper}
	We discuss the role of conditions (RC1) and (RC2) in our context.
	In fact, in the light of Lemma \ref{lem: G has orbits induced by H}, the only reason to introduce these hypothesis is that, under these regularity assumptions, one can show that the action of $G$ on $\A(X)$ has orbits in $\A_X(A)$ induced by $H$ (Lemma \ref{lem: G has orbits induced by H}). This in turns allows to consider $G$-invariant cochains in the relative setting (Lemma \ref{lemma: invariant cochains}).
	In fact, since the action of $G$ on $\A(X)$ does not preserve $\A_X(A)$, some care is needed when passing to invariant cochains.
	Assume, on the contrary, that one could work with $G$-invariant chains regardless of any regularity condition of the action of $G$ on $\A_X(A)$.
	This assumptions seem implicit in Kuessner's work on relative bounded cohomology via multicomplexes (see for example \cite[Corollary 2]{kuessner2015multicomplexes}).
	In this case, using the same argument above, one could prove that the comparison map $\comp^n\colon H^n_b(X,A)\rightarrow H^n(X,A)$ vanishes for every $n\geq \mult(\U)$, where $\U$ is \emph{any} amenable open cover of $X$.
	This is obviously false: if $(X,A)=(M,\partial M)$ for some smooth $m$-manifold $M$ with non-empty boundary, then we know that $M$ is homotopy equivalent a subcomplex of dimension at most $m-1$, therefore it admits a contractible (hence amenable) open cover of cardinality at most $m$ \cite[Remark 2.8]{capovilla2022amenable}; by the standard duality between bounded cohomology and simplicial volume \cite[Proposition 7.10]{frigerio2017bounded}, it would follow that $\|M,\partial M\|=0$. This gives a contradiction, for example, if the interior of $M$ admits a complete finite-volume hyperbolic metric \cite{gromov1982volume}.
\end{rem}

\section{Proof of Theorem \ref{thm: relative vanishing theorem: rel mult and nerves}}
\label{sec: proof of relative mult and nerves theorem}
Let $(X,A)$ be a triangulable pair such that the kernel of the morphism $\pi_1(A\hookrightarrow X, x)$ is amenable for every $x \in A$. 
Let $\U$ be an amenable open cover of $X$ by path-connected subsets satisfying (RC1) and (RC2). We keep the notation from the previous section.

\subsection{Proof of (1)}
We assume that $\U$ is \emph{weakly-convex on $A$}.
Using the fact that $\U$ is locally-finite, up to subdividing $T$, we may suppose that $V_i\cap A\neq \emptyset$ for every $i \in I$ such that $U_i\cap A \neq \emptyset$.
Let $n\geq \mult_A(\U)$. We need to show that the comparison map $\comp^n\colon H^n_b(X,A)\rightarrow H^n(X,A)$ vanishes.
Let $\alpha \in H^n_b(X,A)$. Recall that we have the isomorphism $\Psi^n$,
hence we can take a bounded coclass $\beta \in H^n_b(\A(X),\A_X(A))$ such that $\Psi^n(\beta)=\alpha$.
Let $z \in C^n_b(\A(X),\A_X(A))$ be a cocycle representing $\beta$. As in in the previous section, we can assume that $z$ is alternating and $G$-invariant. By Lemma \ref{lemma: vanishing on simplicial complex and comparison map}, we need to show that $z$ vanishes on every algebraic $n$-simplex of $C_n(T)\subseteq C_n(\A(X))$.

Let $(\Delta,(x_0,\dots, x_n))$ be an algebraic $n$-simplex in $C_n(T)$. 
Since $z$ is alternating, we can assume that the points $x_0,\dots, x_n$ are pairwise-distinct.
We set $i_j = i(x_j)$ for every $j \in \{0,\dots, n\}$.
If there exist $h\neq k \in \{0,\dots, n\}$ such that $i_h=i_k$, then we can argue as in the previous section to show that $z(\Delta,(x_0,\dots, x_n))=0$. We assume therefore that $i_0,\dots i_n$ are pairwise-distinct. 

Since $n\geq \mult_A(\U)$ and since by construction $\bigcap_{j =0}^n U_{i_j}\neq \emptyset$, we deduce that $U_{i_0}\cap\dots\cap U_{i_n}\cap A$ is nonempty. 
Hence, since $\U$ is weakly-convex on $A$, the path-connected component of $U_{i_0}\cap\dots\cap U_{i_n}$ containing $\Delta$ intersects $A$.
Therefore, if $x$ denotes a barycenter of $\Delta$, there exists $x' \in U_{i_0}\cap\dots\cap U_{i_n}\cap A$ and a path $\lambda\colon [0,1]\rightarrow U_{i_0}\cap\dots\cap U_{i_n}$ from $x'$ to $x$.
Let $e_{hk}$ denote the oriented 1-simplex of $\Delta$ from $x_h$ to $x_k$. 
We denote by $\delta_j\colon [0,1]\rightarrow U_{i_0}\cap\dots\cap U_{i_n}$ a path from $x_j$ to $x$ supported in $\Delta$ such that $\delta_h*\bar{\delta}_k$ is homotopic in $X$ relative to the endpoints to a parametrization of $e_{hk}$.
Since $V_{i_j}\cap A \neq \emptyset$, for every $j \in \{0,\dots, n\}$ there exists $x_j' \in V_{i_j}\cap A$. Moreover, by (RC1), $U_{i_j}\cap A$ is path-connected, hence we can find continuous paths $\xi_j\colon [0,1]\rightarrow U_{i_j}\cap A$ from $x_j'$ to $x'$.
For $0\leq h < k \leq n$, we then set $\xi_{hk} = \xi_h * \bar{\xi}_k \colon [0,1]\rightarrow A$.
Since $x_0',\dots, x_n'$ are pairwise-distinct, by definition of $\A_X(A)$, there exists a unique oriented 1-simplex $e_{hk}'$ of $\A_X(A)$ whose projection is a path in $A$ which is homotopic to $\xi_{hk}$ in $X$ relative to the endpoints.
\begin{lemma}
	\label{lem: there exists a simplex in A_X(A)}
	There exists a simplex $\Delta'$ of $\A_X(A)$ whose 1-skeleton is given by the union of the $e_{hk}'$.
\end{lemma}
\begin{proof}
	Notice that the $e_{hk}'$ can be considered to be 1-simplices of $\A(A)$. Of course, the loop $\xi_{jh}*\xi_{hk}*\bar{\xi}_{jk}$ is null-homotopic in $A$ (hence in $X$). By \cite[Proposition 3.33]{FM23} the concatenation of oriented simplices $e_{jh}'*e_{hk}'*\bar{e}_{jk}'$ is null-homotopic in $|\A(A)|$. By \cite[Proposition 3.30]{FM23}, it follows that there exists a unique $n$-simplex $\Delta''$ of $\A(A)$ whose 1-skeleton is the union of the $e_{hk}'$. In conclusion, $\Delta'$ denotes just the image of $\Delta''$ in $\A_X(A)$.
\end{proof}

For every $j \in \{0,\dots, n\}$ we set $\gamma_j=\xi_j * \lambda * \bar{\delta}_j$. By construction $\gamma_j$ joins $x_j'$ with $x_j$ and is supported on $U_{i_j}$.
Therefore, if we set $g_j = \{\gamma_j\}$ if $x_j = x_j'$, and $g_j=\{\gamma_j,\bar{\gamma}_j\}$ if $x_j\neq x_j'$, then $g=\oplus_{j=0}^n g_j$ defines an element of $G$. 
It is clear that  $g\cdot e_{hk}' = e_{hk}$. Since simplices of $\A(X)$ are determined by their 1-skeleton \cite[Proposition 3.30]{FM23}, we obtain $g\cdot \Delta' = \Delta$.
Hence we have that $z(\Delta,(x_0,\dots, x_n)) = z(\Delta',(x_0',\dots, x_n'))$, since $z$ is $G$-invariant, and $z(\Delta',(x_0',\dots, x_n'))=0$, since $z$ vanishes on simplices of $\A_X(A)$. This concludes the proof of (1).

\subsection{Proof of (2)}
Assume that $\U$ is \emph{convex}.
We denote by $N(\U)$ the nerve of $\U$ and by $N(\U_A)$ the nerve of the open cover of $A$ induced by $\U$. This is a well-defined subcomplex of $N(\U)$ under our assumptions.
Given a partition of unity subordinated to $\U$, one can construct a nerve map $\nu\colon X\rightarrow |N(\U)|$, which is unique up to homotopy and which induces a well-defined map of pairs $\nu \colon (X,A)\rightarrow (|N(\U)|,|N(\U_A)|)$.
By assumption, we know that $\U$ is locally-finite.
Therefore, up to taking further subdivisions of $T$, we may suppose that $V_i\neq \emptyset$, for every $i \in I$, and $V_i\cap A\neq \emptyset$ for every $i \in I$ such that $U_i\cap A \neq \emptyset$.
We need to show that there exists a map $\Theta^n\colon H^n_b(X,A)\rightarrow H^n(N(\U),N(\U_A))$ such that the following diagram is commutative:
\begin{equation}
	\label{eq: diagram nerve theorem}
	% https://tikzcd.yichuanshen.de/#N4Igdg9gJgpgziAXAbVABwnAlgFyxMJZABgBpiBdUkANwEMAbAVxiRAAkA9MAfQCMAFAA1SAQQCUIAL6l0mXPkIoAjOSq1GLNlzDCxkmXOx4CRMsvX1mrRB24CAcgIA6zgKrjSDnqJfvxBrIgGMaKRKoW1FZatjoCAD5Orh7xpIk+fikG6jBQAObwRKAAZgBOEAC2SGQgOBBIqhrWbK4AxpVo3NJBZZXV1HVIAExRmjYgrgAqABYwOHRdhiC9VYgjtfWIAMyjzbZtBHndJeWrOxsNuzF2uq5gTJLUDHR8MAwACvImSiClWHnTHDSChSIA
	\begin{tikzcd}
		{H^n_b(X,A)} \arrow[r, "\comp^n"] \arrow[d, "\Theta^n"] & {H^n(X,A)}                                      \\
		{H^n(N(\U),N(\U_A))} \arrow[r, "\cong"]                 & {H^n(|N(\U)|,|N(\U_A)|).} \arrow[u, "H^n(\nu)"']
	\end{tikzcd}
\end{equation}
Recall from the previous sections that the relative bounded cohomology of $(\A(X), \A_X(A))$ may be computed by the complex $C^\bullet_b(\A(X),\A_X(A))^{G}_\alt$ of $G$-invariant alternating cochains which vanish on simplices supported on $\A_X(A)$ (see Lemma \ref{lemma: invariant cochains} and Lemma \ref{lem: G has orbits induced by H}).
Therefore, in order to define $\Theta^n$, we first construct chain maps $\Omega_X^\bullet$ and $\Omega_A^\bullet$ such that the following diagram is commutative
\begin{equation}
	\label{eq: diagram nerve theorem, new}
	% https://tikzcd.yichuanshen.de/#N4Igdg9gJgpgziAXAbVABwnAlgFyxMJZABgBpiBdUkANwEMAbAVxiRAGEA9AHW4CMmDBjBwB9PgApeAQQkANAJQLOAcVG9GOEAF9S6TLnyEUARnJVajFmy68BQkeKndpouROlLV67pp16QDGw8AiIyEwt6ZlZEDh5+QWExSQA5ZwBVJX99YKMiMwjqKOtY2wSHZIkU0VleTIUdCxgoAHN4IlAAMwAnCABbJDIQHAgkE10u3oHEACZqEaQAZgmQHv7B+dHZoqsYkF4AeT6YFro3ePsk7NWpsc2lnei2Q+PTmovEkUbtIA
	\begin{tikzcd}
		C^\bullet_b(\A(X))^G_\alt \arrow[r] \arrow[d, "\Omega_X^\bullet"] & C^\bullet_b(\A_X(A))^H_\alt \arrow[d, "\Omega_A^\bullet"] \\
		C^\bullet_b(N(\U)) \arrow[r]                                      & C^\bullet_b(N(\U_A)),                                    
	\end{tikzcd}
\end{equation}
where the horizontal arrows are restriction maps.
The map $\Omega_X^n$ can be constructed \emph{verbatim} as in \cite[Section 6.4]{FM23}. 
The construction of $\Omega_A^n$ goes as follows.
We identify the set of vertices of $N(\U_A)$ with those indices $i \in I$ such that $U_i \cap A \neq \emptyset$.
Let $z \in C^n_b(\A_X(A))^H_\alt$ and let $(i_0,\dots, i_n) \in N_A(\U)$.
If there exists $h\neq k \in \{0,\dots, n\}$ such that $i_h=i_k$, we set $\Omega_A^n(z)(i_0,\dots, i_n)=0$.
Otherwise, by definition of the nerve $N(\U_A)$, we have that $U_{i_0}\cap \dots \cap U_{i_n} \cap A \neq \emptyset$, hence we may choose a point $q \in U_{i_0}\cap \dots \cap U_{i_n} \cap A$.
Moreover, since $V_{i_j}\cap A \neq \emptyset$ for every $j \in \{0,\dots, n\}$, we can pick a point $v_{i_j}\in V_{i_j}\cap A$. We know by (RC1) that $U_{i_j}\cap A$ is path-connected. Hence there exists a path $\alpha_j \colon [0,1]\rightarrow U_{i_j}\cap A$ from $v_{i_j}$ to $q$.
For $0\leq h< k \leq n$, we set $\alpha_{hk}=\alpha_h *\bar{\alpha}_k$.
Since $v_{i_0},\dots, v_{i_n}$ are pairwise distinct, by definition of $\A_X(A)$, there exists a unique oriented 1-simplex $e_{hk}$ of $\A_X(A)$ whose projection on $A$ is a path which is homotopic to $\alpha_{hk}$ in $X$ relative to the endpoints.
Using the same argument of Lemma \ref{lem: there exists a simplex in A_X(A)}, there exists a unique $n$-simplex $\Delta$ of $\A_X(A)$ whose 1-skeleton is given by the union of the $e_{hk}$.
We then set $\Omega^n_A(z)(i_0,\dots, i_n)=z(\Delta, (v_{i_0},\dots, v_{i_n}))$.

We need to show that $\Omega^n_A(z)$ is well-defined, i.e. different choices in the construction lead to the same value for $z(\Delta, (v_{i_0},\dots, v_{i_n}))$. Let $q' \in U_{i_0}\cap \dots \cap U_{i_n} \cap A$, let $v_{i_j}'\in V_{i_j}\cap A$, $j \in \{0,\dots, n\}$, and let $\alpha_j' \colon [0,1]\rightarrow U_{i_j}\cap A$ be a path from $v_{i_j}'$ to $q'$. We set $\alpha_{hk}'=\alpha_h' *\bar{\alpha}_k'$ and we take $e_{hk}'$ to be the unique edge of $\A_X(A)$ whose projection is a path in $A$ which is homotopic to $\alpha_{hk}'$ in $X$ relative to the endpoints. Finally, we take $\Delta'$ to be the $n$-simplex of $\A_X(A)$ whose 1-skeleton is given by the $e_{hk}'$.
Since $U_{i_0}\cap \dots \cap U_{i_n}$ is path-connected (by convexity of $\U$), there exists a path $\beta \colon [0,1]\rightarrow \bigcap_{j=0}^n U_{i_j}$ from $q$ to $q'$. 
By construction, the path $\gamma_j\coloneqq \alpha_j * \beta * \bar{\alpha}_j$ is supported on $U_{i_j}$ and joins $v_{i_j}$ with $v_{i_j}'$, which are both points of $A$. 
By the regularity conditions (RC1) and (RC2), we deduce from Lemma \ref{lemma: open covers with RC1 and RC2} that, for every $j \in \{0,\dots, n\}$, there exists a path $\lambda_j\colon [0,1]\rightarrow U_{i_j}\cap A$ which is homotopic to $\gamma_j$ in $X$ relative to the endpoints. 
Therefore, if we set $h_j = \{\lambda_j\}$ if $\lambda_j$ is a loop, and $h_j=\{\lambda_j,\bar{\lambda}_j\}$ otherwise, then $h=\oplus_{j=0}^n h_j$ defines an element of $H$.
It is straightforward to check that $h\cdot e_{hk}'=e_{hk}$ for every $h\neq k$, which implies that $h\cdot \Delta' = \Delta$. 
Since $z$ is $H$-invariant, it follows that $z(\Delta, (v_{i_0},\dots, v_{i_n}))= z(\Delta', (v_{i_0}',\dots, v_{i_n}'))$ i.e. that $\Omega_A^n$ is indeed well-defined.

It is easy to check that (\ref{eq: diagram nerve theorem, new}) is commutative and that $\Omega_A^\bullet$ defines a chain map.
From (\ref{eq: diagram nerve theorem, new}), we get a chain map \[\Omega^\bullet\colon C^\bullet_b(\A(X),\A_X(A))^G_\alt \rightarrow C^\bullet(N(\U),N(\U_A)),\] whose induced map in cohomology, precomposed with the inverse of $\Psi^n$, defines $\Theta^n$.
The commutativity of (\ref{eq: diagram nerve theorem}) can be checked \emph{verbatim} as in the absolute case (see \cite[Section 6.4]{FM23}). This concludes the proof of (2).

\bibliographystyle{alpha}
\bibliography{multicomplexes}

\end{document}